\documentclass[10pt]{article}

\usepackage{amssymb}   
\usepackage{amsthm}    
\usepackage{amsmath}   
\usepackage{stmaryrd}  
\usepackage{titletoc}  
\usepackage{mathrsfs}  
\usepackage{graphicx}

\newlength{\defbaselineskip}
\newcommand{\setlinespacing}[1]%
           {\setlength{\baselineskip}{#1 \defbaselineskip}}

\theoremstyle{plain}
\newtheorem{thm}{Theorem}[section]
\newtheorem{cor}[thm]{Corollary}
\newtheorem{lem}[thm]{Lemma}
\newtheorem{prop}[thm]{Proposition}

\theoremstyle{definition}
\newtheorem{defn}{Definition}[section]

\newtheorem{rmk}{Remark}[section]

\newcommand{\eps}{\varepsilon}

\newcommand{\cU}{\mathcal{U}}

\newcommand{\bF}{\mathbb{F}}

\newcommand{\bR}{\mathbb{R}}

\newcommand{\bL}{\mathbb{L}}
\newcommand{\sF}{\mathscr{F}}
\newcommand{\sL}{\mathscr{L}}
\newcommand{\sS}{\mathscr{S}}
\newcommand{\sB}{\mathscr{B}}
\newcommand{\sM}{\mathscr{M}}
\newcommand{\sH}{\mathscr{H}}
\newcommand{\la}{\langle}
\newcommand{\ra}{\rangle}

\newcommand{\rrow}{\rightarrow}

\newcommand{\ves}{\varepsilon}

\makeatletter\@addtoreset{equation}{section} \makeatother
\begin{document}

\title{Maximum Principle for Optimal Control of Neutral Stochastic Functional Differential Systems}
\author{Wenning Wei\footnotemark[1]}
\footnotetext[1]{Department of Finance and Control Sciences, School
of Mathematical Sciences, and Laboratory of Mathematics for
Nonlinear Science, Fudan University, Shanghai 200433, China.
\textit{E-mail}: \texttt{wnwei@fudan.edu.cn}.}

\maketitle

\begin{abstract}
In this paper, the optimal control problem of neutral stochastic functional differential equation (NSFDE) is discussed. A class of so-called neutral backward stochastic functional equations of Volterra type (VNBSFEs) are introduced as the adjoint equation. The existence and uniqueness of VNBSFE is established. The Pontryagin maximum principle is constructed for controlled NSFDE with Lagrange type cost functional.
\end{abstract}



\section{Introduction}
In this paper, $\delta\geq0$ is a constant and $T>0$ is the terminal time. Let $(\Omega,\sF,\bF,P)$ be a complete filtered probability space on which a $d$-dimensional Brownian motion $W=\big\{W(t):t\in[0,T]\big\}$ is defined. $\big\{\sF_t:t\in[0,T]\big\}$ is the natural filtration of $W$ augmented by all $P$-null sets in $\sF$. Define
$\sF_t:=\sF_0$ for any $t\in[-\delta,0]$. Then $\bF:=\big\{\sF_t:t\in[-\delta,T]\big\}$ is a filtration satisfying the usual conditions.

Consider the following stochastic optimal control problem: minimize the Lagrange type cost functional
$$J(u(\cdot)):=E\Big[\int_0^Tl(t,X^t,u(t))dt\Big],$$
subject to
\begin{equation}\label{NSFDEs}
\left\{
\begin{split}
&d\,\big[X(t)-g(t,X^t,u(t))\big]=b(t,X^t,u(t))dt+\sigma(t,X^t,u(t))dW(t),\quad t\in[0,T],\\
&X(t)=\phi(t),~~t\in[-\delta,0],
\end{split}
\right.
\end{equation}
where $X^t$ denotes the restriction of the path of $X$ on $[t-\delta,t]$, and $u(\cdot)$ is the control variable. We will establish the maximum principle of this optimal control problem. As we know, it is the first time to consider this problem.

In many applications, people model systems via differential equations, and assume that the evolution rate of the state is
independent of the past state and determined solely at the present, such as the ordinary differential equations and partial differential equations. However, under closer scrutiny, a more realistic model would include some of the past state of the system. That is, the evolution rate of state should depend not only on the present state, but also some of the past state, or more generally, it should depend not only on the past and present state, but also on the evolution rate of the state in the past. In stochastic term, it can be expressed by the neutral stochastic functional differential equations (NSFDEs):
\begin{equation*}
\left\{
\begin{split}
&d\bigl[X(t)-g(t,X^t)\bigr]=b(t,X^t)\,dt+\sigma(t,X^t)\,dW(t),\quad t\in[0,T];\\
&X(t)=\phi(t),~~t\in[-\delta,0].
\end{split}
\right.
\end{equation*}
If $g\equiv0$, it is a stochastic functional differential equation (SFDE). When choosing $\delta\geq T$, and $g(t,\cdot), b(t,\cdot)$ and $\sigma(t,\cdot)$ suitably, It contains the interesting case that $g(t,\cdot), b(t,\cdot)$ and $\sigma(t,\cdot)$ depend on $X$ on $[0,t]$. Neutral functional differential equations model a large class of system with after-effect, which are widely used in biology, mechanics, physics, medicine and economics, such as population sizes, commodity supply fluctuations and so on. See \cite{Kolmanovskii99, Hale77, Hale_73, Kolmanovskii_Nosov81, Kolmanovskii1986, Cushing_77} and reference therein.

By now the research on NSFDEs mostly focuses on the well-posedness and stability of the solutions, see \cite{Shaikhe_10, Kolmanovskii_02, Shaikhet_97, Mao95, Mao_09, Ren_09} and reference therein. The optimal control problem of deterministic neutral functional differential equation was discussed by \cite{Banks72, Kent71, Kolmanovskii69}. The maximum principle of controlled SFDE was discussed by Hu and Peng \cite{Hu_Peng_96, Wu_09}. For the best knowledge of the author, the maximum principle of controlled NSFDE is still open. The difficulty of this problem mainly relies on the adjoint equation. As we know, adjoint equation is crucial for constructing maximum principle. The solution of NSFDE would not be a semi-martingale due to the part $g(t,X^t)$ in the left hand. Therefore, the traditional method dealing with the optimal control problem to SDEs introduced by Bismut \cite{Bismut73, Bismut76, Bismut78} will not apply. In this paper, under a technical condition $(A3)$, we introduce a linear neutral backward stochastic functional equation of Volterra type (VNBSFE) as the adjoint equation. The general form of VNBSFE goes as following:
\begin{equation}\label{VNBSFEs}
\left\{
\begin{split}
&Y(t)-G(t,Y_t)=\Psi(t)+\int_t^Tf(t,s,Y_s,Z(t,s),Z(s,t;\delta))\,ds-\int_t^TZ(t,s)\,dW(s),~t\in[0,T];\\
&Y(t)=\xi(t),~~t\in(T,T+\delta],
\end{split}
\right.
\end{equation}
where $Y_t$ denotes the restriction of the path of $Y$ on $[t,t+\delta]$, $Z(\cdot,\cdot)$ is defined on $[0,T+\delta]\times[0,T+\delta]$, and $Z(s,t;\delta)$ denotes the restriction of $Z$ on $[s,s+\delta]\times[t,t+\delta]$. When $G\equiv0$ and $\delta=0$, it reduces to
$$Y(t)=\Psi(t)+\int_t^Tf(t,s,Y(s),Z(t,s),Z(s,t))\,ds-\int_t^TZ(t,s)\,dW(s),$$
which is discussed by Yong \cite{Yong08} and called backward stochastic Volterra integral equation (BSVIE).

Similar as Yong \cite{Yong08}, the definition of $M$-solution of VNBSFE \eqref{VNBSFEs} is introduced. Then we prove the existence and uniqueness of VNBSFE \eqref{VNBSFEs} and give an estimate. Via the solution of a linear VNSFDE, we construct the maximum principle of the optimal control problem of NSFDE \eqref{NSFDEs}. When the state equation reduces to a stochastic differential equation (SDE), the maximum principle here will not degenerate to the traditional one in Bismut \cite{Bismut78}. We compare this two maximum principles and establish the explicit relation between them.

The rest of this paper is organized as follow: In section 2, we introduce some notations and the optimal control problem. Section 3 is devoted to the duality of linear NSFDEs and linear VNBSFEs. In section 4, the existence and uniqueness of VNBSFE \eqref{VNBSFEs} are proved. In section 5, we construct the maximum principle for controlled NSFDE \eqref{NSFDE} with Lagrange type cost functional. As an example, when the state equation reduces to a SDE, we compare the maximum principle here with the traditional one in \cite{Bismut78}, and establish the explicit relation between them in section 6.

\section{Preliminaries}
In this section, we will introduce some notations and the optimal control problem.

\subsection{Notations}
For any $A$ being a vector or matrix, denote $A'$ as the transformation of $A$. Denote $H$ as some Euclidean space, such as $\bR^n,\bR^{n\times d},~etc.$, and $|\cdot|$ as the norm and $\la\cdot,\cdot\ra$ as the inner product in $H$. Define
$$\bL^2(\Omega;H):=\left\{\eta:\Omega\rrow H, \hbox{ \rm $\sF_T$-measurable }|~E[|\eta|^2]<+\infty\right\},$$
For $r,s,\tau,\nu\in[-\delta,T+\delta]$, define
$$\sL^2_{\bF}(r,s;H):=\Big\{\theta:[r,s]\times\Omega\rrow H, \hbox{ \rm $\bF$-adapted} \bigm|
                 E\int_r^s|\theta(u)|^2\,du<+\infty \Big\},$$
$$\sS^2_{\bF}([r,s];H):=\Big\{\theta:[r,s]\times\Omega\rrow H, \hbox{ \rm $\bF$-adapted, path-continuous}
           \bigm| E\sup_{r\leq u\leq s}|\theta(u)|^2<+\infty\Big\},$$
$$\bL^2(r,s;\bL^2(\Omega;H)):=\Big\{\psi:[r,s]\times\Omega\rrow H, \hbox{ \rm $\sB([r,s])\times\sF_T$-measurable}
                  \bigm| E\!\int_r^s\!\!|\psi(u)|^2du<+\infty\Big\},$$
\begin{equation*}
\begin{split}
&\bL^2(r,s;\bL^2(\tau,\nu;H)):=\Big\{v:[r,s]\!\times\![\tau,\nu]\rrow H, \hbox{ \rm jointly-measurable} \bigm| \int_r^s\!\!\int_{\tau}^{\nu}\!\!|v(t,s)|^2ds\,dt<+\infty\Big\},\\
&\bL^2(r,s;\sL^2_{\bF}(\tau,\nu;H)):=\Big\{\vartheta:[r,s]\times[\tau,\nu]\times\Omega\rrow H
      \hbox{ \rm is $\sB([r,s]\times[\tau,\nu])\times\sF_T$-measurable},~~~~\\
&~~~~~~~~~~~~~~~~~~~~~~~~\hbox{ \rm $\vartheta(t,\cdot)$ is $\bF$-adapted for all $t\in[r,s]$, }
    \bigm| E\int_r^s\!\int_{\tau}^{\nu}\!|\vartheta(t,s)|^2dsdt<+\infty\Big\}.
\end{split}
\end{equation*}
For simplicity, denote
$$\sH^2(r,s):=\sL^2_{\bF}(r,s;\bR^n)\times\bL^2(r,s;\sL^2_{\bF}(r,s;\bR^{n\times d})),$$
equipped with norm
$$\|(\theta,\vartheta)\|^2_{\sH^2(r,s)}=E\left[\int_r^s|\theta(u)|^2du+\int_r^s\int_r^s|\vartheta(\nu,u)|^2\,du\,d\nu\right],$$
and
$$\sM^2(r,s):=\Big\{(\theta,\vartheta)\in\sH^2(r,s)~\Bigm|~\theta(t)=E[\theta(t)]+\int_r^s\vartheta(r,u)dW(u),~\forall
t\in[r,s]\Big\}.$$

\subsection{The Optimal Control Problem}
Consider a controlled NSFDE,
\begin{equation}\label{NSFDE}
\left\{
\begin{split}
&d\,\bigl[X(t)-g(t,X^t)\bigr]=b(t,X^t,u(t))\,dt+\sigma(t,X^t,u(t))\,dW(t), \quad t\in[0,T],\\
&X(t)=\phi(t),\quad t\in[-\delta,0],
\end{split}
  \right.
\end{equation}
where $X^t$ denote the restriction of $X$ on $[t-\delta,t]$, $\phi\in\sS^2_{\bF}([-\delta,0];\bR^n)$,
$$g:[0,T]\times\Omega\times C([0,\delta];\bR^n)\rrow\bR^n,$$
$$b:[0,T]\times\Omega\times C([0,\delta];\bR^n)\times\bR^m\rrow\bR^n,$$
$$\sigma:[0,T]\times\Omega\times C([0,\delta];\bR^n)\times\bR^m\rrow\bR^{n\times d},$$
are jointly measurable, and  $g(\cdot,\psi)$, $b(\cdot,\psi,u)$ and $\sigma(\cdot,\psi,u)$ are $\bF$-progressively
measurable for any $(\psi,u)\in C([0,\delta];\bR^n)\times\bR^m$. Here, for simplicity, we only discuss the case that $g$ does not depend $u$. For $g$ depends on $u$, see section 5.

Let $U\subseteq\bR^m$ be a nonempty convex set. Denote
$$\cU_{ad}:=\Big\{u(\cdot):[0,T]\times\Omega\rrow U,\hbox{ \rm $\bF$-progressively measurable }| ~E\int_0^T|u(t)|^2\,dt<+\infty\Big\}$$
as the admissible control set. For any $u(\cdot)\in\cU_{ad}$, denote the cost functional as
$$J(u(\cdot))=E\Big[\int_0^T l(t,X^t,u(t))\,dt\Big],$$
where
$$l:[0,T]\times\Omega\times C([0,\delta];\bR^n)\times\bR^m\rrow\bR$$
is jointly measurable, and $l(\cdot,\psi,u)$ is $\bF$-progressively measurable for any $(\psi,u)\in C([0,\delta];\bR^n)\times\bR^m$. Define $l(t,\cdot,\cdot)\equiv0$, for any $t\in(T,T+\delta]$.

Our optimal problem is to find a control $\bar{u}(\cdot)\in\cU_{ad}$, such that
$$J(\bar{u}(\cdot))=\inf_{u(\cdot)\in\cU_{ad}}\,J(u(\cdot)).$$
Denote $\bar{X}(\cdot)$ as the solution of NSFDE \eqref{NSFDE} corresponding to $\bar{u}(\cdot)$. Then $(\bar{X}(\cdot),\bar{u}(\cdot))$ is called the optimal pair.

Here are some assumptions on the coefficients.

(A1) $b,\sigma,l,g$ are continuously Fr\'{e}chet differentiable with respect to $x$, and $b,\sigma,l$ are continuously differentiable with respect to $u$. The derivatives $b_x,b_u\,\sigma_x,\sigma_u\,l_x,l_u,\,g_x$ are all bounded.

(A2) $b(\cdot,0,0),\sigma(\cdot,0,0),l(\cdot,0,0)\in\sL^2_{\sF}(0,T;H)$, $H=\bR^n,\bR^{n\times d},\bR$ respectively. $g,\,g_x$ are both continuous in $t$, and there is a constant $0<\kappa<1$, such that $\|g_x\|\leq\kappa$.

\medskip
Via the standard argument in \cite{Kolmanovskii_Nosov81, Mao95} among others, it is not too hard to show that under assumptions $(A1)$ and $(A2)$, for any $\phi(\cdot)\in\sS^2_{\bF}([-\delta,0];\bR^n)$ and $u(\cdot)\in\cU$, NSFDE \eqref{NSFDE} admits a unique solution $X(\cdot)\in\sS^2_{\bF}([-\delta,T];\bR^n)$. Thus the cost function $J(u(\cdot))$ is well-defined.

\section{Adjoint Equation}
Suppose that $(\bar{X}(\cdot),\bar{u}(\cdot))$ is the optimal pair. Let $\chi(\cdot)$ be the solution of the following Linear NSFDE,
\begin{equation}\label{1-order-derivative}
\left\{
  \begin{split}
    &d[\chi(t)-\bar{g}_x(t)\chi^t]=[\bar{b}_x(t)\chi^t+\bar{b}_u(t)\bar{v}(t)]\,dt
                        +[\bar{\sigma}_x(t)\chi^t+\bar{\sigma}_u(t)\bar{v}(t)]\,dW(t),\quad t\in[0,T],\\
    &\chi(t)=0,\quad t\in[-\delta,0].
  \end{split}
\right.
\end{equation}
Denote
$$I(\chi(\cdot))=E\int_0^T\bar{l}_x(t)\chi^t\,dt$$
as a linear functional.
Here $\bar{g}_x(t):=g_x(t,\bar{X}^t)$, $\bar{b}_x(t):=b_x(t,\bar{X}^t,\bar{u}(t))$, $\bar{b}_u(t):=b_u(t,\bar{X}^t,\bar{u}(t))$, $\bar{\sigma}_x(t):=\sigma_x(t,\bar{X}^t,\bar{u}(t))$, $\bar{\sigma}_u(t):=\sigma_u(t,\bar{X}^t,\bar{u}(t))$ and $\bar{l}_x(t):=l_x(t,\bar{X}^t,\bar{u}(t))$. By (A1) and (A2), $\bar{g}_x,\bar{b}_x,\bar{\sigma}_x,\bar{l}_x$ are the linear functionals of $C([0,\delta];\bR^n)$.

Denote
$$V_0([0,\delta];\bR^{k\times n}):=\{f:[0,\delta]\rrow\bR^{k\times n} \hbox{ \rm is bounded variational and left continuous on } [0,\delta)\}.$$
Via the Riesz Representation theorem, we have the following lemma.
\begin{lem}
There exist $G(t,\cdot),\,B(t,\cdot)\in V_0([0,\delta];\bR^{n\times n})$, $\Sigma_i(t,\cdot)\in V_0([0,\delta];\bR^{d\times n})(i=1,\cdot,d)$ and $L(t,\cdot)\in V_0([0,\delta];\bR^{1\times n})$, such that for all $\phi\in C([0,T];\bR^n)$,
$$\bar{g}_x(t)\phi=\int_0^{\delta}G(t,dr)\phi(r),\quad \bar{b}_x(t)\phi=\int_0^{\delta}B(t,dr)\phi(r),$$
$$\bar{\sigma}_x(t)\phi=\int_0^{\delta}\Sigma(t,dr)\phi(r),\quad \bar{l}_x(t)\phi=\int_0^{\delta}L(t,dr)\phi(r).$$
\end{lem}

We need the following technical assumption:

(A3) There exist probability measures $\lambda_i(i=0,1,2,3)$ on $[0,\delta]$, and $\bar{G}(t,r), \bar{B}(t,r), \bar{\Sigma}(t,r), \bar{L}(t,r)$, such that,
$$\int_0^{\delta}G(t,dr)\phi(r)=\int_0^{\delta}\bar{G}(t,r)\phi(r)\lambda_0(dr),
\quad \int_0^{\delta}B(t,dr)\phi(r)=\int_0^{\delta}\bar{B}(t,r)\phi(r)\lambda_1(dr),$$
$$\int_0^{\delta}\Sigma(t,dr)\phi(r)=\int_0^{\delta}\bar{\Sigma}(t,r)\phi(r)\lambda_2(dr),
\quad \int_0^{\delta}L(t,dr)\phi(r)=\int_0^{\delta}\bar{L}(t,r)\phi(r)\lambda_3(dr),$$

\begin{rmk}
Assumption (A3) holds in many cases, for example,

$\bullet$ $\bar{g}_x(t),\bar{b}_x(t),\bar{\sigma}_x(t),\bar{L}_x(t)$ are all deterministic and continuous in t, (the proof is similar to Lemma 4.1 in Hu and Peng \cite{Hu_Peng_96})

$\bullet$ In NSFDE \eqref{NSFDE}, $g(t,X^t)=\hat{g}(t,\int_0^{\delta}\alpha(t,r)X(t-r)\lambda_0(dr))$, and $b,\sigma,l$ possess similar form.
\end{rmk}

Under assumption (A3), equation \eqref{1-order-derivative} and the linear functional reduce to
\begin{equation}\label{linear_NSFDE}
  \begin{split}
\chi(t)-\int_0^{\delta}\bar{G}(t,r)\chi(t-r)\lambda_0(dr)
 &=\int_0^t\Big[\int_0^{\delta}\bar{B}(s,r)\chi(s-r)\lambda_1(dr)+\bar{b}_u(s)v(s)\Big]\,ds\\
    &+\int_0^t\Big[\int_0^{\delta}\bar{\Sigma}(s,r)\chi(s-r)\lambda_2(dr)+\bar{\sigma}_u(s)v(s)\Big]dW(s),
  \end{split}
\end{equation}
    and
$$I(\chi(\cdot))=E\int_0^T\int_0^{\delta}\bar{L}(t,r)\chi(t-r)\lambda_3(dr)\,dt.$$

\medskip
Denote $\rho(t):=\int_0^t\bar{b}_u(s)\,ds+\int_0^t\bar{\sigma}_u(s)\,dW(s)$. We have the following duality.
\begin{prop}\label{dual_represt}
Let $\chi\in\sS^2_{\bF}([-\delta,T];\bR^n)$ be the solution of NSFDE \eqref{linear_NSFDE}, and $(Y,Z)\in\sM^2(0,T+\delta)$ be the adapted M-solution of the following linear VNBSFE:
\begin{equation}\label{linear_VNBSFE}
\left\{
\begin{split}
&Y(t)-E_t\Big[\!\int_0^{\delta}\!\bar{G}'(t\!+\!r,r)Y(t\!+\!r)\lambda_0(dr)\!\Big]
 =\int_0^\delta\!\bar{L}(t+r,r)\lambda_3(dr)
 +\int_t^T\!\!E_s\Big[\int_0^{\delta}\!\!\!\bar{B}'(t\!+\!r,r)Y(s\!+\!r)\lambda_1(dr) \\
&~~~~~~~~~~~~~~~~~~~~~~
   +\int_0^{\delta}\!\!\!\bar{\Sigma}'(t\!+\!r,r)Z(s\!+\!r,t\!+\!r)\lambda_2(dr)\Big]\,ds-\int_t^TZ(t,s)\,dW(s),~~t\in[0,T],\\
&Y(t)=0,~~t\in(T,T+\delta].
\end{split}\right.
\end{equation}
Then the following relation holds:
$$I(\chi(\cdot))=E\Big[\int_0^T\!\la\rho(t),\,Y(t)\ra dt\Big].$$
\end{prop}

Note that the well-posedness of VNBSFE \eqref{linear_VNBSFE} will be discussed in the next section. Here we assume that \eqref{linear_VNBSFE} holds for $(Y,Z)\in\sM^2(0,T+\delta)$.

\begin{proof}
In view of \eqref{linear_NSFDE}, we have
\begin{equation*}
\begin{split}
\rho(t)=&\chi(t)-\int_0^{\delta}\!\!\bar{G}(t,r)\chi(t\!-\!r)\lambda_0(dr)
               -\int_0^t\!\!\int_0^{\delta}\!\!\bar{B}(s,r)\chi(s\!-\!r)\lambda_1(dr)\,ds\\
        &-\int_0^t\!\int_0^{\delta}\!\!\bar{\Sigma}(s,r)\chi(s\!-\!r)\lambda_2(dr)\,dW(s).
\end{split}
\end{equation*}
Since $Y\in\sL^2_{\bF}(0,T;\bR^n)$, we have
\begin{equation*}
\begin{split}
 &E\Big[\int_0^T\la Y(t),\rho(t)\ra\,dt\Big]\\
=&E\Big[\int_0^T\la Y(t),\chi(t)\ra\,dt\Big]
      -E\Big[\int_0^T\Big\la Y(t),\,\int_0^{\delta}\!\!\bar{G}(t,r)\chi(t\!-\!r)\lambda_0(dr)\Big\ra\,dt\Big]\\
&-E\Big[\int_0^T\Big\la Y(t),\,\int_0^t\!\!\int_0^{\delta}\!\!\bar{B}(s,r)\chi(s-r)\lambda_1(dr)ds\Big\ra\,dt\Big]\\
 &-E\Big[\int_0^T\!\!\Big\la\!Y(t),\,\int_0^t\!\!\int_0^{\delta}\!\!\bar{\Sigma}(s,r)\chi(s\!-\!r)\lambda_2(dr)dW(s)\!\Big\ra dt\Big]\\
=&E\Big[\int_0^T\la Y(t),\chi(t)\ra\,dt\Big]+I_1+I_2+I_3
\end{split}
\end{equation*}
By Fubini's theorem, we have
\begin{equation}\label{p1}
\begin{split}
  I_1=&\int_0^{\delta}\int_0^T\bigl\la Y(t),\,\bar{G}(t,r)\chi(t-r)\bigr\ra\,dt\,\lambda_0(dr)\\
     =&\int_0^{\delta}\int_{-r}^{T-r}\bigl\la \bar{G}'(t+r,r)Y(t+r),\,\chi(t)\bigr\ra\,dt\,\lambda_0(dr)\\
     =&\int_0^{\delta}\int_0^T\bigl\la\bar{G}'(t+r,r)Y(t+r),\chi(t)\bigr\ra\,dt\,\lambda_0(dr)\\
     =&\int_0^T\left\la\int_0^{\delta}\bar{G}'(t+r,r)Y(t+r)\lambda_0(dr),\chi(t)\right\ra dt,
\end{split}
\end{equation}
 and
\begin{equation}\label{p3}
\begin{split}
   I_2=&\int_0^{\delta}\int_0^T\left\la \bar{B}'(t,r)\!\!\int_t^T\!\!\!Y(s)ds,\,\chi(t-r)\right\ra\,dt\,\lambda_1(dr)\\
      =&\int_0^{\delta}\int_{-r}^{T-r}\left\la \bar{B}'(t+r,r)\!\!\int_t^T\!\!Y(s+r)ds,\,\chi(t)\right\ra\,dt\,\lambda_1(dr)\\
      =&\int_0^T\left\la\int_t^T\!\!\int_0^{\delta}\!\bar{B}'(t+r,r)Y(s+r)\lambda_1(dr)ds,\,\chi(t)\right\ra\,dt.\\
\end{split}
\end{equation}
Since $(Y,Z)\in\sM^2(0,T+\delta)$, $Y(t)=E[Y(t)]+\int_0^tZ(t,s)\,dW(s)$. Then
\begin{equation}\label{p5}
\begin{split}
  I_3=&\int_0^T\left\la\int_0^tZ(t,s)dW(s),\,\int_0^t\!\!\int_0^{\delta}\!\!\bar{\Sigma}(s,r)\chi(s-r)\lambda_2(dr)dW(s)\right\ra\,dt\\
     =&\int_0^T\int_0^t\left\la Z(t,s),\,\int_0^{\delta}\!\!\bar{\Sigma}(s,r)\chi(s-r)\lambda_2(dr)\right\ra\,ds\,dt\\
     =&\int_0^{\delta}\int_0^T\left\la\int_t^T \!\!\bar{\Sigma}'(t,r)Z(s,t)ds,\,\chi(t-r)\right\ra\,dt\,\lambda_2(dr)\\
     =&\int_0^{\delta}\int_{-r}^{T-r}\left\la \bar{\Sigma}'(t+r,r)\!\int_t^T\!\! Z(s+r,t+r)ds,\,\chi(t)\right\ra\,dt\,\lambda_2(dr)\\
     =&\int_0^T \left\la\int_t^T\!\!\int_0^{\delta}\!\!\bar{\Sigma}'(t+r,r)Z(s+r,t+r)\lambda_2(dr)ds,\,\chi(t)\right\ra\,dt.
\end{split}
\end{equation}
Since for any $t\in(T,T+\delta]$, $l(t,\cdot,\cdot)\equiv0$, then $\bar{L}(t,r)\equiv0$, for any $t\in(T,T+\delta]$. Deduce from \eqref{p1},\eqref{p3} and \eqref{p5}, we have
\begin{equation*}
\begin{split}
     E\Big[\int_0^T\bigl\la Y(t),\rho(t)\bigr\ra\,dt\Big]
 =\,&E\Big[\int_0^T\Big\la \int_0^\delta\!\!\bar{L}(t\!+\!r,r)\lambda_3(dr)
              -\int_t^T\!\!Z(t,s)dW(s),\,\chi(t)\Big\ra\,dt\Big]\\
 =\,&E\Big[\int_0^T\Big\la \int_0^\delta\!\!\bar{L}(t\!+\!r,r)\lambda_3(dr),\,\chi(t)\Big\ra\,dt\Big]\\
 =\,&E\Big[\int_0^T\int_0^\delta\bar{L}(t,r)\chi(t-r)\,\lambda_3(dr)\,dt\Big]=I(\chi(\cdot)).
\end{split}
\end{equation*}
\end{proof}

\section{Well-posedness of VNBSFEs}
In this section, we are concerned with the well-posedness of general VNBSFEs. We introduce the definition of adapted M-solution, which was first introduced in Yong \cite{Yong08} for backward stochastic Volterra integral equations (BSVIEs). The existence, uniqueness and an estimate of the adapted $M$-solution of VNBSFE are proved.

Consider a general VNBSFE,
\begin{equation}\label{VNBSFE}
\left\{\!
\begin{split}
Y(t)-G(t,Y_t)&=\Psi(t)+\!\int_t^T\!\!\!f(t,s,Y_s,Z(t,s),Z(s,t;\delta))ds+\!\int_t^T\!\!\!Z(t,s)dW(s),~t\in[0,T];\\
Y(t)=\xi(t),~~&~~t\in(T,T+\delta],
\end{split}
\right.
\end{equation}
where $Y_t$ denotes the restriction of $Y$ on $[t,t+\delta]$, $Z(\cdot,\cdot)$ is an unknown mapping defined on $[0,T+\delta]\times[0,T+\delta]$, $Z(t,s)$ denotes the value of $Z$ at $(t,s)$, and $Z(s,t;\delta)$ denotes the restriction of $Z(\cdot,\cdot)$ on $[s,s+\delta]\times[t,t+\delta]$.

For any $0\leq R\leq S\leq T+\delta$, define
$$\triangle[R,S]:=\{(t,s)\in[R,S]\times[R,S]~|~R\leq t\leq s\leq S\},$$
$$\triangle^c[R,S]:=[R,S]\times[R,S]\setminus\triangle[R,S].$$
For simplicity, denote $\triangle:=\triangle[0,T]$, $\triangle^c:=\triangle^c[0,T]$,
$\triangle_{\delta}:=\triangle[0,T+\delta]$ and $\triangle^c_{\delta}:=\triangle^c[0,T+\delta]$.

\bigskip
$(G,f)$ in \eqref{VNBSFE} is called the generator of VNBSFEs. For $(G,f)$, there exist two functional $J$ and $F$,

$\bullet$ $J:[0,T]\times\Omega\times\bL^2(0,\delta;\bR^n)\rrow\bR^n$ is jointly measurable, and $J(\cdot,\phi)$ is $\bF$-progressively measurable for any $\phi\in\bL^2(0,\delta;\bR^n)$,

$\bullet$ $F:\triangle\times\Omega\times\bL^2(0,\delta;\bR^n)\times\bR^{n\times d}\times\bL^2(0,\delta;\bL^2(0,\delta;\bR^{n\times d}))\rrow\bR^n$ is jointly measurable, and $F(t,\cdot,\phi,w,\varphi)$ is $\bF$-progressively measurable for all $(t,\phi,w,\varphi)$ fixed in corresponding space, and $(J,F)$ satisfies,

(H1) There are $\kappa\in[0,1)$ and $\varrho_0$ being a probability measure on $[0,\delta]$, such that for any
$\phi,\bar{\phi}\in\bL^2(0,\delta;\bR^n)$,
\begin{equation}\label{G}
|J(t,\phi_t)-J(t,\bar{\phi}_t)|^2\leq \kappa\int_0^{\delta}|\phi(u)-\bar{\phi}(u)|^2\,\varrho_0(du),
\end{equation}

(H2) There are $L>0$, $\varrho_1$ and $\varrho_2$ being probability measures on $[0,\delta]$, such that for all
$(\phi,w,\varphi),(\bar{\phi},\bar{w},\bar{\varphi})$ in corresponding space and $(t,s)\in\triangle$
\begin{equation}\label{f}
\begin{split}
        &~~|F(t,s,\phi,w,\varphi))-f(t,s,\bar{\phi},\bar{w},\bar{\varphi})|\\
\leq\,\,&L\,\Big[\int_0^{\delta}|\phi(u)-\bar{\phi}(u)|\,\varrho_1(du)+|w-\bar{w}|
    +\int_0^{\delta}|\varphi(u,u)-\bar{\varphi}(u,u)|\,\varrho_2(du)\Big].
\end{split}
\end{equation}

$(G,f)$ are the functionals defined by
$$G(t,y_t)=E_t[J(t,y_t)],\quad f(t,s,y_s,z(t,s),z(s,t;\delta))=E_s[F(t,s,y_s,z(t,s),z(s,t;\delta))],$$
for all $(y(\cdot),z(\cdot))\in\sH^2(0,T+\delta)$, $(t,s)\in\triangle$.

(H3) $|G(t,0)|\in\sL^2_{\bF}(0,T;\bR^n)$,  $f_0(t,s):=f(t,s,0,0,0,0)\in\bL^2(0,T;\sL^2_{\bF}(0,T;\bR^n))$.

\medskip
Here is the definition of adapted $M\textmd{-}$solution of VNBSFE \eqref{VNBSFE}.

\begin{defn}\label{M-solution}
A pair of process $(Y,Z)\in\sH^2(0,T+\delta)$ is called the adapted $M\textmd{-}$solution of VNBSFE \eqref{VNBSFE}, if \eqref{VNBSFE} holds in It\^{o}'s sense for almost all $t\in[0,T+\delta]$,
$$Y(t)=E[Y(t)]+\int_0^tZ(t,s)\,dW(s), ~~~a.e.~t\in[0,T+\delta],$$
and $Z(t,s)=0$ on $(t,s)\in\triangle_{\delta}\setminus\triangle$.
\end{defn}

\begin{rmk}
In VNBSFE \eqref{VNBSFE}, we only set the terminal condition $Y(t)=\xi(t)$ on $(T,T+\delta]$. The value of $Y$ at $T$ is
determined via
$$Y(T)+G(T,Y_T)=\Psi(T),$$
and the value of $Z$ on $\triangle^c_{\delta}\setminus\triangle^c$ is endogenously determined by
\begin{equation}\label{3}
\xi(t)=E[\xi(t)]+\int_0^tZ(t,s)\,dW(s),~~t\in[T,T+\delta].
\end{equation}
Since $f$ depends on $Z(t,s)$ on $\triangle$ without anticipation, the equality of \eqref{VNBSFE} is independent of the value of $Z$ on $\triangle_{\delta}\setminus\triangle$. That is, any value of $Z$ on $\triangle_{\delta}\setminus\triangle$ equalizes VNBSFE \eqref{VNBSFE}. For the uniqueness of solution, we define
$Z(t,s)=0$ in Definition \ref{M-solution} on $\triangle_{\delta}\setminus\triangle$.
\end{rmk}

For all $\tau\in[0,T+\delta]$, define a subspace of $\sH^2(0,\tau)$,
$$\sM^2(0,\tau):=\left\{(\theta,\vartheta)\in\sH^2(0,\tau)~\Bigm|
~\theta(t)=E[\theta(t)]+\int_0^t\vartheta(t,s)dW(s),~\forall t\in[0,\tau]\right\}$$
equipped with norm
$$\|(\theta,\vartheta)\|^2_{\sM^2(0,\tau)}=E\left[\int_0^\tau|\theta(u)|^2du
+\int_0^\tau\int_s^\tau|\vartheta(s,u)|^2\,du\,ds\right].$$
Then $\sM^2(0,\tau)$ is a closed subspace of $\sH^2(0,\tau)$ under the norm $\|\cdot\|_{\sH^2(0,\tau)}$. In fact, It is also a complete space under $\|\cdot\|_{\sM^2(0,\tau)}$, because $\|\cdot\|_{\sH^2(0,\tau)}$ is equivalent to $\|\cdot\|_{\sM^2(0,\tau)}$ in $\sM^2(0,\tau)$. For all
$(\theta,\vartheta)\in\sM^2(0,\tau)$, viewing
$$\theta(t)=E[\theta(t)]+\int_0^t\vartheta(t,s)\,dW(s),~~t\leq\tau,$$
we have
$$E\left[\int_0^t|\vartheta(t,s)|^2\,ds\right]=E\bigl[|\theta(t)-E[\theta(t)]|^2\bigr]\leq 2E[|\theta(t)|^2].$$
Then
\begin{equation*}
\begin{split}
&\|(\theta,\vartheta)\|^2_{\sH^2(0,\tau)}=E\left[\int_0^{\tau}|\theta(t)|^2\,dt
            +\int_0^{\tau}\int_0^{\tau}|\vartheta(t,s)|^2\,ds\,dt\right]\\
=&E\left[\int_0^{\tau}|\theta(t)|^2\,dt+\int_0^{\tau}\int_t^{\tau}|\vartheta(t,s)|^2\,ds\,dt
        +\int_0^{\tau}\int_0^t|\vartheta(t,s)|^2\,ds\,dt\right]\\
\leq&2E\left[\int_0^{\tau}|\theta(t)|^2\,dt+\int_0^{\tau}\int_t^{\tau}|\vartheta(t,s)|^2\,ds\,dt\right]
        =2\|(\theta,\vartheta)\|^2_{\sM^2(0,\tau)}\\
\leq&2E\left[\int_0^{\tau}|\theta(t)|^2\,dt+\int_0^{\tau}\int_0^{\tau}|\vartheta(t,s)|^2\,ds\,dt\right]
           =2\|(\theta,\vartheta)\|^2_{\sH^2(0,\tau)}.
\end{split}
\end{equation*}

Therefore, if $(Y,Z)\in\sH^2(0,T+\delta)$ is the adapted $M$-solution of VNBSFE \eqref{VNBSFE}, it means $(Y,Z)\in\sM^2(0,T+\delta)$ and $Z(t,s)=0$ on $\triangle_{\delta}\setminus\triangle$.

Before showing the existence and uniqueness of adapted M-solution of VNBSFE \eqref{VNBSFE}, we discuss
some backward equations. First, consider the following backward stochastic differential equation (BSDE),
\begin{equation}\label{BSDE}
y(t)=\zeta+\int_t^Tf(s)\,ds-\int_t^Tz(s)\,dW(s),~~~t\in[0,T],
\end{equation}
where $f:[0,T]\times\Omega\rrow \bR^n$ is $\bF$-progressively measurable, and $f(t)\in\sL^2_{\bF}(0,T;\bR^n)$ and $\zeta\in\bL^2(\Omega;\bR^n)$. Then we have
\begin{lem}\label{BSDE_conclusion}
BSDE \eqref{BSDE} admits a unique pair of solution
$(y,z)\in\sS^2_{\bF}([0,T];\bR^n)\times\sL^2_{\bF}(0,T;\bR^{n\times d})$, and for any $t\in[0,T]$,
\begin{equation}\label{BSDE_est}
e^{\beta t}|Y(t)|^2+E\Big[\int_t^Te^{\beta s}|Z(s)|^2\,ds\Bigm|\sF_t\Big]
\leq E\Big[2e^{\beta T}|\zeta|^2+\alpha\int_t^Te^{\beta s}|f(s)|^2\,ds\Bigm|\sF_t\Big],
\end{equation}
where $\alpha$ and $\beta$ are any two positive constants satisfying $\beta>\frac{2}{\alpha}$.
\end{lem}

Consider the following backward integral equation,
\begin{equation}\label{BSDE_sequence}
\rho(t,s)=\Phi(t)+\int_s^Th(t,u)\,du-\int_s^T\nu(t,u)\,dW_u,~~~~~~t\in[0,T],
\end{equation}
where $\Phi\in\bL^2(0,T;\bL^2(\Omega;\bR^n))$, $h:[0,T]^2\times\Omega\rrow\bR^n$ are given, and
$h\in\bL^2(0,T;\sL^2_{\bF}(0,T;\bR^n))$. Fixed $t\in[0,T]$, equation \eqref{BSDE_sequence} is a BSDE with generator $h(t,\cdot)\in\sL^2_{\bF}(0,T;\bR^n)$ and terminal condition $\Phi(t)\in\bL^2(\Omega;\bR^n)$. So \eqref{BSDE_sequence} is a family of BSDEs parameterized by $t\in[0,T]$. Let $s=t$, $y(t)=\rho(t,t)$ and $z(t,u)=\nu(t,u)$ when $u\geq t$. Then
\begin{equation}\label{volterraBSIE}
y(t)=\Phi(t)+\int_t^Th(t,u)\,du-\int_t^Tz(t,u)\,dW_u,~~t\in[0,T].
\end{equation}
It is not a BSDE, but a backward stochastic Volterra integral equation (BSVIE), which was first discussed in Lin
\cite{Lin_jianzhong02}.

\begin{rmk}
In equation \eqref{volterraBSIE}, the equality is independent of $z$ on $\triangle^c$. Therefore any value of $z$ on
$\triangle^c$ equalizes \eqref{volterraBSIE}, such as $z(t,u)=\nu(t,u)$ or $z(t,u)=0$, $(t,u)\in\triangle^c$. Therefore, the uniqueness of equation \eqref{volterraBSIE} does not hold. However, in the definition of adapted $M$-solution, the value of $z$ on $\triangle$ is settled by $y(t)=E[y(t)]+\int_0^tz(t,s)\,dW(s)$. This determines the uniqueness.
\end{rmk}

The following lemma can be found in Yong \cite{Yong08}.
\begin{lem}\label{conclusionVBSIE}
BSVIE \eqref{volterraBSIE} admits a unique pair of solution $(y(\cdot),z(\cdot,\cdot))\in\sM^2(0,T)$. In addition,
\begin{equation}\label{11}
E\Big[e^{\beta t}|y(t)|^2+\int_t^Te^{\beta s}|z(t,s)|^2\,ds\Big]
\leq E\Big[|\Phi(t)|^2+\alpha\int_t^Te^{\beta s}|h(t,s)|^2\,ds\Big].
\end{equation}
where $\alpha>0$ and $\beta>\frac{2}{\alpha}$ are any two positive constants.
\end{lem}

\medskip
The following theorem is devoted to the existence and uniqueness of adapted M-solution and an estimation of VNBSFE \eqref{VNBSFE}.

\begin{thm}\label{exist_unique_thm}
Suppose that $(G,f)$ satisfies $(H1)$-$(H3)$. Then for any $\Psi(\cdot)\in\bL^2(0,T;\bL^2(\Omega;\bR^n))$ and
$\xi(\cdot)\in\sL^2_{\bF}(T,T+\delta;\bR^n)$, VNBSFE \eqref{VNBSFE} admits a unique pair of adapted $M$-solution
$(Y,Z)\in\sM^2(0,T+\delta)$. Moreover the following estimate holds:
\begin{equation}\label{esti_solution}
\begin{split}
      &~~~E\left[\int_0^{T+\delta}\!|Y(t)|^2\,dt+\int_0^{T+\delta}\!\int_t^{T+\delta}\!|Z(t,s)|^2\,ds\,dt\right]\\
\leq\,&C\,E\bigg[\!\int_0^T\!|\Psi(t)|^2dt+\int_T^{T+\delta}\!|\xi(t)|^2dt
         +\int_0^T\!|G(t,0)|^2dt+\int_0^T\!\int_t^T\!|f_0(t,s)|^2dsdt\bigg].
\end{split}
\end{equation}
\end{thm}

\begin{proof}

\textbf{Step 1}: Let us define a subset of $\sM^2(0,T+\delta)$:
$$\sM^2_{\xi}(0,T):=\big\{(\theta,\vartheta)\in\sM^2(0,T+\delta)~|~\theta(t)=\xi(t),~\forall t\in(T,T+\delta],~\textrm{and}~\vartheta(t,s)=0,~\forall (t,s)\in\triangle_{\delta}\setminus\triangle\big\}$$
equipped with the norm
$$\|(\theta,\vartheta)\|^2=E\Big[\int_0^T e^{\beta t}|\theta(t)|^2\,dt +\int_0^T\int_t^T e^{\beta s}|\vartheta(t,s)|^2\,ds\,dt\Big],$$
where $\beta$ is a positive constant which will be specified in Step 2. It is obvious that $\sM^2_{\xi}(0,T)$ is closed.

For any $(y(\cdot),z(\cdot))\in\sM^2_{\xi}(0,T)$, consider
\begin{equation}\label{equ1}
\left\{\begin{array}{l}
\begin{split}
&Y(t)-G(t,y_t)=\Psi(t)+\int_t^Tf(t,s,y_s,z(t,s),z(s,t;\delta))\,ds-\int_t^TZ(t,s)\,dW_s,~~t\in[0,T],\\
&Y(t)=\xi(t),~~t\in(T,T+\delta].
\end{split}
\end{array}\right.
\end{equation}
Denote $\tilde{Y}(t):=Y(t)-G(t,y_t)$, then
\begin{equation}\label{1BSVIE}
\tilde{Y}(t)=\Psi(t)+\int_t^Tf(t,s,y_s,z(t,s),z(s,t;\delta))\,ds-\int_t^TZ(t,s)\,dW_s,~~t\in[0,T].
\end{equation}
Since $(y,z)\in\sM^2_{\xi}(0,T)$, $f(t,s,y_s,z(t,s),z(s,t;\delta))\in\bL^2(0,T;\sL^2_{\bF}(0,T;\bR^{n\times m}))$ via (H1) and (H2).

Via Lemma \ref{conclusionVBSIE}, \eqref{1BSVIE} admits a unique pair of solution $(\tilde{Y},Z)\in\sM^2(0,T)$. Define
\begin{equation*}
 Y(t)=\left\{
\begin{split}
&\tilde{Y}(t)-G(t,y(t)),~~~~~t\in[0,T];\\
&\xi(t),~~~~~~~~~~~~~~~~~~~t\in(T,T+\delta].
\end{split}
\right.
\end{equation*}
Then $Y\in\sL^2_{\bF}(0,T+\delta;\bR^n)$. Define $Z(t,s)=0$ on $\triangle_{\delta}\setminus\triangle$ and modify the value of $Z$ on $\triangle_{\delta}$ such that
$$Y(t)=E[Y(t)]+\int_0^tZ(t,s)\,ds,~~\forall t\in[0,T+\delta].$$
Then $(Y,Z)\in\sM^2_{\xi}(0,T)$ is an adapted $M$-solution of equation \eqref{equ1}.

\textbf{Step 2}: Consider the mapping $\Gamma:(y(\cdot),z(\cdot,\cdot))\mapsto(Y(\cdot),Z(\cdot,\cdot))$
with $(Y,Z)$ in Step 1. We prove that $\Gamma$ is a contraction.

Take another pair $(\bar{y}(\cdot),\bar{z}(\cdot))\in \sM^2_{\xi}(0,T)$, and denote $(\bar{Y}(\cdot),\bar{Z}(\cdot))\in \sM^2_{\xi}(0,T)$ as the adapted $M$-solution of \eqref{equ1} with $(y(\cdot),z(\cdot))$ replaced by $(\bar{y}(\cdot),\bar{z}(\cdot))$. Define $\Delta Y(t):=Y(t)-\bar{Y}(t)$, $\Delta
Z(t,s):=Z(t,s)-\bar{Z}(t,s)$, $\Delta y(t):=y(t)-\bar{y}(t)$ and
$\Delta z(t):=z(t,s)-\bar{z}(t,s)$. Then
\begin{equation*}
\begin{split}
 &~~~\Delta Y(t)-[G(t,y_t)-G(t,\bar{y}_t)]\\
=&\int_t^T\!\!\Bigl[f(t,s,y_s,z(t,s),z(s,t;\delta))-f(t,s,\bar{y}_s,\bar{z}(t,s),\bar{z}(s,t;\delta))\Bigr]\,ds
       +\int_t^T\!\!\Delta Z(t,s)\,dW(s).
\end{split}
\end{equation*}
Denote $C:=C(L,T,n,d)$. It varies from time to time. In view of \eqref{11} in Lemma \ref{conclusionVBSIE} and choosing $\beta=\frac{2}{\alpha}$, we have
\begin{equation}\label{33}
\begin{split}
 &E\Big[e^{\beta t}|\Delta Y(t)-[G(t,y_t)-G(t,\bar{y}_t)]|^2+\int_t^Te^{\beta s}|\Delta Z(t,s)|^2\,ds\Big]  \\
\leq\, &\alpha E\Big[\int_t^Te^{\beta s}
         |f(t,s,y_s,z(t,s),z(s,t;\delta))-f(t,s,\bar{y}_s,\bar{z}(t,s),\bar{z}(s,t;\delta))|^2\,ds\Big]\\
\leq\, &\alpha C E\left\{\int_t^T\!\!e^{\beta s}\!\Big[\int_0^{\delta}\!\!|\Delta y(s+u)|^2\varrho_1(du)+|\Delta z(t,s)|^2
        +\int_0^{\delta}\!\!|\Delta z(s\!+\!u,t\!+\!u)|^2\varrho_2(du)\Big]\,ds\right\}.
\end{split}
\end{equation}
Integrate \eqref{33} in t from 0 to T, and denote $\Delta G(t):=G(t,y_t)-G(t,\bar{y}_t)$,
\begin{equation}\label{34}
\begin{split}
       &~~~~E\Big[\int_0^T\!\!e^{\beta t}|\Delta Y(t)-\Delta G(t)|^2\,dt
                 +\int_0^T\!\!\int_t^Te^{\beta s}|\Delta Z(t,s)|^2\,ds\,dt\Big]  \\
\leq\, &\alpha C E\biggl\{\int_0^T\!\int_t^T \!e^{\beta s}\Big[\int_0^{\delta}\!|\Delta y(s+u)|^2\varrho_1(du)
                 +|\Delta z(t,s)|^2\\
       &~~~~~~~~~~~~~~~~~~~~~~~~~~~~~~~~~~~~~~+\int_0^{\delta}\!|\Delta z(s+u,t+u)|^2\,\varrho_2(du)\Big]dsdt\biggr\}\\
\leq\, &\alpha C E\left[\int_0^T\! e^{\beta s}|\Delta y(s)|^2\,ds
         +\int_0^T\!\int_t^T \!e^{\beta s}\Bigl(|\Delta z(t,s)|^2+|\Delta z(s,t)|^2\Big)dsdt\right]\\
\leq\, &\alpha C\,E\left[\int_0^T\! e^{\beta s}|\Delta y(s)|^2\,ds+\int_0^T\!\int_t^T\! e^{\beta s}|\Delta z(t,s)|^2\,ds\,dt\right].
\end{split}
\end{equation}
The last inequality is due to
$$E\left[\int_0^T\!\!\int_t^T\!e^{\beta s}|\Delta z(s,t)|^2\,dsdt\right]=E\left[\int_0^T\!\!e^{\beta
t}\!\!\int_0^t\!|\Delta z(t,s)|^2\,dsdt\right]\leq E\left[\int_0^T\!\!e^{\beta t}|\Delta y(t)|^2\,dt\right].$$

Since for all $\gamma\in(0,1)$ and $a,b\in\bR^n$, $|a-b|^2\leq(1-\gamma)|a|^2-(\frac{1}{\gamma}-1)|b|^2$, then
$$|\Delta Y(t)-\Delta G(t)|^2 \leq (1-\gamma)|\Delta Y(t)|^2 -(\frac{1}{\gamma}-1)|\Delta G(t)|^2.$$
\eqref{34} reduces to
\begin{equation*}
\begin{split}
      &E\Big[(1-\gamma)\int_0^Te^{\beta t}|\Delta Y(t)|^2\,dt+\int_0^T\int_t^Te^{\beta s}|\Delta Z(t,s)|^2\,ds\,dt\Big]  \\
\leq\,&(\frac{1}{\gamma}-1)E\Big[\!\int_0^T\!\!e^{\beta t}|\Delta G(t)|^2dt\Big]
            +\alpha C\,E\Big[\!\int_0^T\!\!e^{\beta t}|\Delta y(t)|^2dt+\int_0^T\!\!\int_t^T\!\!e^{\beta s}|\Delta z(t,s)|^2dsdt\Big]\\
\leq\,&\Big[\!(\frac{1}{\gamma}\!-\!1)\kappa\!+\!\alpha C\Big]
              E\Big[\int_0^T\!\! e^{\beta t}|\Delta y(t)|^2\,ds\Big]
              +\alpha C\,E\Big[\int_0^T\!\!\!\int_t^T\!\! e^{\beta s}|\Delta z(t,s)|^2\,ds\,dt\Big].
\end{split}
\end{equation*}

To prove $\Gamma$ is a contraction, it suffices to show: for all $\kappa\in(0,1)$, there is $\gamma\in(0,1)$ such
that
$$(\frac{1}{\gamma}-1)\kappa+\alpha C\,<1-\gamma~~~\textrm{and}~~~\alpha C\,<1.$$
which hold true via choosing $\alpha$ small sufficiently.

Therefore $\Gamma$ admits a unique fixed point $(Y,Z)\in\sM^2_{\xi}(0,T)$. Then $(Y,Z)\in\sM^2(0,T+\delta)$ is the unique adapted $M\textmd{-}$solution of VNBSFE \eqref{VNBSFE}.

\textbf{Step 3}:
In view of \eqref{11} in Lemma \ref{conclusionVBSIE}, we have
\begin{equation*}
\begin{split}
      &E\Big[e^{\beta t}|Y(t)-G(t,Y_t)|^2+\int_t^Te^{\beta s}|Z(t,s)|^2\,ds\Big]\\
\leq\,&E\Big[e^{\beta T}|\Psi(t)|^2+\alpha\int_t^Te^{\beta s}|f(t,s,Y_s,Z(t,s),Z(s,t;\delta))|^2\,ds\Big]\\
\leq\,&E\bigg\{e^{\beta T}|\Psi(t)|^2+\alpha C\,\int_t^Te^{\beta
           s}\Bigl[|f_0(t,s)|^2+\int_0^{\delta}|Y(s+u)|^2\varrho_1(du)+|Z(t,s)|^2\\
      &+\int_0^{\delta}|Z(s+u,t+u)|^2\varrho_2(du)\Bigr]\Bigr)\,ds\bigg\}.
\end{split}
\end{equation*}
Similar as the method in Step 2, we have for all $\alpha\in(0,1)$ and $M>0$,
\begin{equation*}
\begin{split}
      &E\Big[(1-\gamma)\int_0^Te^{\beta t}|Y(t)|^2\,dt+\int_0^T\!\!\int_t^Te^{\beta s}|Z(t,s)|^2\,ds\,dt\Big]\\
\leq\,&E\Big[e^{\beta T}\int_0^T|\Psi(t)|^2\,dt+(\frac{1}{\gamma}-1)\int_0^Te^{\beta t}|G(t,Y_t)|^2 \,dt
       +\alpha C\,\int_0^T\!\!\int_t^Te^{\beta s}|f_0(t,s)|^2\,ds\,dt\Big]\\
      &~~~+\alpha C\,E\Big[\int_0^{T+\delta}\!\!e^{\beta t}|Y(t)|^2\,dt
            +\int_0^T\!\int_t^T\!e^{\beta s}|Z(t,s)|^2\,ds\,dt\Big]\\
\leq\,&E\bigg[e^{\beta T}\int_0^T|\Psi(t)|^2dt+(\frac{1}{\gamma}\!-\!1)(1\!+\!M)\int_0^Te^{\beta t}|G(t,0)|^2dt
           +\alpha C\,\int_0^T\!\!\int_t^Te^{\beta s}|f_0(t,s)|^2dsdt\\
      &~+\Bigl[(\frac{1}{\gamma}\!-\!1)(1\!+\!\frac{1}{M})\kappa\!+\!\alpha C\Bigr]\int_0^Te^{\beta t}|Y(t)|^2\,dt
          +\alpha C\int_0^T\!\int_t^T\!e^{\beta s}|Z(t,s)|^2\,ds\,dt\bigg].\\
\end{split}
\end{equation*}
It is easy to prove that there are $\gamma\in(0,1)$ and $M>0$, such that the following two inequalities hold for any $\kappa\in(0,1)$ via choosing $\alpha$ small sufficiently,
$$(\frac{1}{\gamma}-1)(1+\frac{1}{M})\kappa+\alpha C<1-\gamma~~~~\textrm{and}~~~~\alpha C<1.$$
Then the estimate \eqref{esti_solution} holds.
\end{proof}

\section{Maximum Principle}
In this section, we construct a maximum principle for the optimal control problem in section 2.

Suppose that $(\bar{X}(\cdot),\bar{u}(\cdot))$ is an optimal pair. For any $u(\cdot)\in\cU_{ad}$, denote $v(\cdot):=u(\cdot)-\bar{u}(\cdot)$ and
$$u_{\ves}(\cdot):=\bar{u}(\cdot)+\ves v(\cdot)\in\cU_{ad},\quad \forall \ves\in[0,1].$$
Denote $X_{\eps}(\cdot)$ as the corresponding solution of NSFDE \eqref{NSFDE}.

Before construct the maximum principle, we need some lemmas about the first order expansion. Let $\Gamma$ be a metric space. Consider
\begin{equation*}
\left\{
  \begin{split}
    &d[y_{\gamma}(t)-\mathcal {G}(t,\gamma,y_{\gamma}^t)]=\mathcal {B}(t,\gamma,y^t_{\gamma})\,dt+\mathcal {R}(t,\gamma,y^t_{\gamma})\,dW_t,\quad t\in[0,T],\\
    &y_{\gamma}(t)=\varphi(t),\quad t\in[-\delta,0],
  \end{split}
  \right.
\end{equation*}
where $\varphi\in C([-\delta,0];\bR^n)$,
$$\mathcal {G},\mathcal {B}:[0,T]\times\Omega\times\Gamma\times C([0,\delta];\bR^n)\rrow\bR^n,$$
$$\mathcal {R}:[0,T]\times\Omega\times\Gamma\times C([0,\delta];\bR^n)\rrow\bR^{n\times d}.$$
For any $\gamma\in\Gamma$ fixed, $\mathcal {G}(\cdot,\gamma,\phi),\mathcal {B}(\cdot,\gamma,\phi),\mathcal {R}(\cdot,\gamma,\phi)$ are $\bF$-progressively measurable, for any $\phi\in C([0,\delta];\bR^n)$, and satisfy the following assumptions.

(i) $\mathcal {G}(\cdot,\gamma,\phi)$ is continuous in $t$ and $\mathcal {G}(\cdot,\gamma,0)\in\sS^2_{\bF}([0,T];\bR^n)$. For any $\phi_1,\phi_2\in C([0,\delta];\bR^n)$, there is $\kappa\in(0,1)$, such that
$$\|\mathcal {G}(t,\gamma,\phi_1)-\mathcal {G}(t,\gamma,\phi_2)\|\leq \kappa \|\phi_1-\phi_2\|.$$

(ii) $\mathcal {B}(\cdot,\gamma,0), \mathcal {M}(\cdot,\gamma,0)\in\sL^2_{\bF}([0,T];\bR^n)$. For any $\phi_1,\phi_2\in C([0,\delta];\bR^n)$, there is $L>0$, such that
$$\|\mathcal {B}(t,\gamma,\phi_1)-\mathcal {B}(t,\gamma,\phi_2)\|+\|\mathcal {R}(t,\gamma,\phi_1)-\mathcal {R}(t,\gamma,\phi_2)\|\leq L \|\phi_1-\phi_2\|.$$

(iii) For $\gamma_0\in\Gamma$,
$$\lim_{\gamma\rrow\gamma_0}E\sup_{t\in[0,T]}|\mathcal {G}(t,\gamma,y^t_{\gamma_0})-\mathcal {G}(t,\gamma_0,y^t_{\gamma_0})|=0,$$
$$\lim_{\gamma\rrow\gamma_0}E\int_0^T|\mathcal {B}(t,\gamma,y^t_{\gamma_0})-\mathcal {B}(t,\gamma_0,y^t_{\gamma_0})|\,dt=0,$$
$$\lim_{\gamma\rrow\gamma_0}E\int_0^T|\mathcal {R}(t,\gamma,y^t_{\gamma_0})-\mathcal {R}(t,\gamma_0,y^t_{\gamma_0})|\,dt=0.$$

\begin{lem}\label{convergence}
  Under the above assumptions, we have
$$\lim_{\gamma\rrow\gamma_0}E\sup_{t\in[0,T]}|y_{\gamma}(t)-y_{\gamma_0}(t)|^2=0.$$
\end{lem}

\begin{proof}
  It is easy to prove via Gronwall's inequality and the method in the proof of Theorem \ref{exist_unique_thm}.
\end{proof}

Via this lemma, we can deduce the following first-order expansion.
\begin{lem}\label{expansion}
Suppose that (A1) and (A2) hold. Then we have the following first order expansion,
$$X_{\ves}(t)=\bar{X}(t)+\ves \chi(t)+R_{\ves}(t),\quad t\in[0,T],$$
where
$$\lim_{\ves\rrow0^+}\frac{1}{\ves^2}E[\sup_{0\leq t\leq T}|R_{\ves}(t)|^2]=0.$$
\end{lem}
\begin{proof}
Via Lemma \ref{convergence}, it is easy to prove
$$\lim_{\ves\rrow0^+}E\sup_{t\in[0,T]}|X_{\ves}(t)-\bar{X}(t)|^2=0.$$
Set $z_{\ves}:=\frac{X_{\ves}(t)-\bar{X}(t)}{\ves}$, then
\begin{equation*}
  \begin{split}
    &~~z_{\ves}(t)-\int_0^1g_x(t,\bar{X}\!+\!\theta\ves z^t_{\ves})\,z_{\ves}^t\,d\theta\\
    =&\int_0^t\Big[\int_0^1\!\!b_x(s,\bar{X}\!+\!\theta\ves z^t_{\ves},u_{\ves}(s))z_{\ves}^s\,d\theta+
    \int_0^1\!\!b_u(s,\bar{X}^s,\bar{u}(s)\!+\!\theta v(s))\,v(s)\,d\theta\Big]\,ds\\
    &+\int_0^t\Big[\int_0^1\!\!\sigma_x(s,\bar{X}\!+\!\theta\ves z^t_{\ves},u_{\ves}(s))\,z_{\ves}^s\,d\theta+
    \int_0^1\!\!\sigma_u(s,\bar{X}^s,\bar{u}(s)\!+\!\theta v(s))\,v(s)\,d\eta\Big]dW(s).
  \end{split}
\end{equation*}
In view of Lemma \ref{convergence}, we have
$$\lim_{\ves\rrow0^+}E\sup_{t\in[0,T]}|z_{\ves}(t)-\chi(t)|^2=0.$$
\end{proof}

Then we have the maximum principle.
\begin{thm}\label{maximum_principle}
Suppose that (A1)-(A3) hold. Let $(\bar{X}(\cdot),\bar{u}(\cdot))$ be an optimal pair, and $(Y,Z)\in\sM^2(0,T+\delta)$ be the adapted $M$-solution of linear VNBSFE \eqref{linear_VNBSFE}. Then we have for all $~u\in U$,
$$\left\la \bar{l}_u(t)+\bar{b}'_u(t)E_t\Bigl[\int_t^TY(s)\,ds\Bigr]
+\bar{\sigma}_u'(t)E_t\Bigl[\int_t^TZ(s,t)\,ds\Bigr],\,u-\bar{u}(t)\right\ra\geq 0, \quad t\in[0,T]\textrm{-}a.e.,$$
where $\bar{l}_u(t):=l(t,\bar{X}^t,\bar{u}(t))$.
\end{thm}

\begin{proof}
In view of assumptions (A1)-(A3) and Theorem \ref{exist_unique_thm}, VNBSFE \eqref{linear_VNBSFE} admits a unique pair of adapted $M$-solution $(Y,Z)\in\sM^2(0,T+\delta)$. Then,
\begin{equation}
  \begin{split}
  0\leq\,&\frac{J(u_{\ves}(\cdot))-J(\bar{u}(\cdot))}{\ves}
      =E\int_0^T\frac{l(t,X^t_{\ves},u_{\ves}(t))-l(t,\bar{X}^t,\bar{u}(t))}{\ves}\,dt\\
      =\,&E\int_0^T\int_0^1\!\!l_u(t,X^t_{\ves},\bar{u}(t)\!+\!\theta\ves v(t))v(t)d\theta\,dt
           +E\int_0^T\int_0^1\!\!l_u(t,\bar{X}^t\!+\!\theta(X^t_{\ves}\!-\!\bar{X}^t),\bar{u}(t))
           \frac{X^t_{\ves}\!-\!\bar{X}^t}{\ves}d\theta\,dt\\
      =\,&E\int_0^T\big[\bar{l}_u(t)v(t)+\bar{l}_x(t)\chi^t\big]\,dt
           +E\int_0^T\int_0^1\!\!\big[l_u(t,X^t_{\ves},\bar{u}(t)\!+\!\theta\ves v(t))-\bar{l}_u(t)\big]v(t)d\theta\,dt\\
        &+E\int_0^T\Big(\int_0^1\!\!l_u(t,\bar{X}^t\!+\!\theta(X^t_{\ves}\!-\!\bar{X}^t),\bar{u}(t))
            \frac{X^t_{\ves}-\bar{X}^t}{\ves}d\theta-\bar{l}_x(t)\chi^t\Big)\,dt.\\
  \end{split}
\end{equation}
In view of Lemma \ref{expansion}, we have
$$E\int_0^T\big[\bar{l}_u(t)v(t)+\bar{l}_x(t)\chi^t\big]\,dt\geq\,0.$$
Then applying assumption (A3) and Proposition \ref{dual_represt}, we have
\begin{equation*}
\begin{split}
0\leq\,&E\int_0^T\Big(\int_0^{\delta}\bar{L}(t,r)\chi(t-r)\lambda_3(dr)+\la\bar{l}_u(t),v(t)\ra\Big)\,dt\\
    =\,&E\int_0^T\Big(\bigl\la Y(t),\,\rho(t)\bigr\ra+\bigl\la\bar{l}_u(t),\,v(t)\bigr\ra\Big)\,dt\\
    =\,&E\Big[\int_0^T\bigl\la\bar{l}_u(t),\,v(t)\bigr\ra\,dt
                      +\int_0^T\!\!\left\la Y(t),\,\int_0^t\bar{b}_u(s)v(s)\,ds\right\ra\,dt\Big]\\
       &~~+E\int_0^T\left\la Y(t),\,\int_0^t\bar{\sigma}_u(s)v(s)\,dW(s)\right\ra\,dt\\
    =\,&E\int_0^T\left\la\bar{l}_u(t)+\bar{b}'_u(t)\int_t^T\!\!Y(s)ds
    +\bar{\sigma}'_u(t)\int_t^T\!\!Z(s,t)ds,\,u(t)-\bar{u}(t)\right\ra dt.
\end{split}
\end{equation*}
The last equality is due to
\begin{equation*}
\begin{split}
   &\int_0^T\!\!\left\la Y(t),\,\int_0^t\!\!\bar{\sigma}_u(s)\Delta u(s)dW(s)\right\ra\,dt\\
=\,&\int_0^T\!\!\left\la\int_0^t\!\!Z(t,s)dW(s),\,\int_0^t\bar{\sigma}_u(s)\Delta u(s)dW(s)\right\ra\,dt\\
=\,&\int_0^T\!\!\int_0^t\!\!\big\la Z(t,s),\,\bar{\sigma}_u(s)\Delta u(s)\big\ra\,dsdt
=\int_0^T\!\!\left\la\bar{\sigma}'_u(t)\int_t^T\!\!Z(s,t)ds,\,\Delta u(t)\right\ra\,dt.
\end{split}
\end{equation*}
So we have for all $u\in U$ and almost all $t\in[0,T]$,
$$\left\la\bar{l}_u(t)+\bar{b}'_u(t)E_t\int_t^T\!\!Y(s)ds
+\bar{\sigma}'_u(t)E_t\int_t^T\!\!Z(s,t)ds,\,u-\bar{u}(t)\right\ra\,\geq0.$$

\end{proof}

\begin{rmk}
If $g$ depends on $u$, assume that $g$ is continuously differentiable in $u$ with bounded derivative, and $\bar{g}_u(t):=g(t,\bar{X}^t,\bar{u}(t))$ is continuously in t. Define the admissible control set as follow,
$$\cU_{ad}:=\{u:[0,T]\times\Omega\rrow U,\hbox{ \rm path-continuous and bounded, $\bF$-progressively measurable}\}.$$
Define
$$\rho(t):=\bar{g}_u(t)v(t)+\int_0^t\bar{b}_u(s)v(s)\,ds+\int_0^t\bar{\sigma}_u(s)v(s)\,dW(s).$$
The maximum principle can be derived similarly. That is, for all $u\in U$ and almost all $t\in[0,T]$,
$$\left\la\bar{g}'_u(t)Y(t)+\bar{l}_u(t)+\bar{b}'_u(t)E_t\left[\int_t^T\!\!Y(s)ds\right]
    +\bar{\sigma}'_u(t)E_t\left[\int_t^T\!\!Z(s,t)ds\right],\,u-\bar{u}(t)\right\ra\,\geq0,$$
and
$$\left\la\bar{g}'_u(0)E\int_0^TY(t)dt,\,u(0)-\bar{u}(0)\right\ra\,\geq0.$$
\end{rmk}

\section{An Example}
In this section, we establish the maximum principle of a controlled stochastic differential equation (SDE) via the method in the preceding sections. Maximum principle for controlled SDEs was first discussed by Bismut \cite{Bismut73, Bismut76, Bismut78}, in which the maximum principle was established via the solution of linear Backward stochastic differential equations (BSDEs). Here we compare the maximum principle here with that in \cite{Bismut78}, and show the explicit relation between them.

If $g\equiv0$ and $\delta=0$, the controlled NSFDE \eqref{NSFDEs} reduces to the following controlled SDE,
\begin{equation*}
\left\{
\begin{split}
&dX(t)=b(t,X(t),u(t))\,dt+\sigma(t,X(t),u(t))\,dW(t),\quad t\in[0,T],\\
&X(0)=x,
\end{split}
\right.
\end{equation*}
and the cost functional reduces to
$$J(u(\cdot))=E\Big[\int_0^T l(t,X(t),u(t))\,dt\Big].$$
Suppose that $(A1)$ and $(A2)$ still hold. The admissible control set and the optimal control problem are the same as Section 2.

As a corollary of Theorem \ref{maximum_principle}, we have the maximum principle,
\begin{cor}\label{cor1}
Suppose that $(Y,Z)\in\sM^2(0,T)$ is the adapted $M$-solution of the following equation:
\begin{equation}\label{VBS}
Y(t)=\bar{l}_x(t)+\int_t^T\Bigl(\bar{b}'_x(t)Y(s)+\bar{\sigma}'_x(t)Z(s,t)\Bigr)\,ds+\int_t^TZ(t,s)\,dW(s).
\end{equation}
Let $(\bar{X}(\cdot),\bar{u}(\cdot))$ be the optimal pair. Then for all $u\in U$,
\begin{equation*}
\left\la\bar{l}_u(t)+\bar{b}'_u(t)E_t\Big[\int_t^T\!\!Y(s)ds\Big]+\bar{\sigma}'_u(t)E_t\Big[\int_t^T\!\!Z(s,t)ds\Big],
\,u-\bar{u}(t)\right\ra \geq0, \quad  t\in[0,T]\textrm{-}a.e..
\end{equation*}
\end{cor}

Recall the maximum principle in Bismut \cite{Bismut78}.
\begin{prop}\label{prop1}
Suppose that $(P,Q)\in\sS^2_{\bF}([0,T];\bR^2)\times\sL^2_{\bF}(0,T;\bR^{n\times d})$ is the solution of the following BSDE:
\begin{equation}\label{BSDEP}
P(t)=\int_t^T\Bigl(\bar{b}_x(s)P(s)+\bar{\sigma}_x(s)Q(s)+\bar{h}_x(s)\Bigr)\,ds-\int_t^TQ(s)\,dW(s).
\end{equation}
Let $(\bar{X}(\cdot),\bar{u}(\cdot))$ be the optimal pair. Then for all $u\in U$,
\begin{equation*}
\bigl\la\bar{l}_u(t)+\bar{b}'_u(t)P(t)+\bar{\sigma}'_u(t)Q(t),\,u-\bar{u}(t)\bigr\ra\geq0,\quad t\in[0,T]\textrm{-}a.e..
\end{equation*}
\end{prop}

In fact, the two maximum principles possess the following relationship.
\begin{thm}\label{relation}
Let $(Y,Z)$ and $(P,Q)$ be processes as above, then
$$P(t)=E_t\Big[\int_t^TY(s)\,ds\Big],~~~~~~Q(t)=E_t\Big[\int_t^TZ(s,t)\,ds\Big].$$
Moreover, the maximum principle in Corollary \ref{cor1} and Proposition \ref{prop1} are equivalent.
\end{thm}

\begin{proof}
Similar to the proof in Theorem \ref{maximum_principle}, we have $\frac{X_{\eps}(\cdot)-\bar{X}(\cdot)}{\eps}$ converges to
$\chi(\cdot)$ in $\sS^2_{\bF}([0,T];\bR^n)$, where $\chi(\cdot)$ satisfies
$$\chi(t)=\int_0^t\Big[\bar{b}_x(s)\chi(s)+\bar{b}_u(s)v(s)\Big]\,ds
+\int_0^t\Big[\bar{\sigma}_x(s)\chi(s)+\bar{\sigma}_u(s)v(s)\Big]\,dW(s).$$
Denote
$$\rho(t):=\int_0^t\bar{b}_u(s)v(s)\,ds+\int_0^t\bar{\sigma}_u(s)v(s)\,dW(s).$$
The duality between linear SDE and BSDE \eqref{BSDEP} shows
\begin{equation}\label{5}
E\int_0^T\la\chi(t),\,\bar{l}_x(t)\ra\,dt=E\int_0^T\Big(\la P(t),\,\bar{b}_u(t)v(t)\ra +\la
Q(t),\,\bar{\sigma}_u(t)v(t)\ra\Big)\,dt,
\end{equation}
and the duality between linear SDE and VNBSFE \eqref{VBS} shows:
\begin{equation}\label{55}
\begin{split}
   &E\int_0^T\la\chi(t),\,\bar{l}_x(t)\ra\,dt =E\int_0^T\la\rho(t),\,Y(t)\ra\,dt\\
=\,&E\int_0^T\!\left\la\int_0^t\bar{b}_u(s)v(s)\,ds,\,Y(t)\right\ra dt
    +E\int_0^T\!\left\la\int_0^t\bar{\sigma}_u(s)v(s)\,dW(s),\,Y(t)\!\right\ra dt\\
=\,&E\int_0^T\!\left\la\bar{b}_u(t)\Delta u(t),\,\!\int_t^T\!\!Y(s)ds\right\ra\,dt
  +E\int_0^T\!\left\la\bar{\sigma}_u(t)v(t),\,\int_t^T\!Z(s,t)ds\right\ra dt.
\end{split}
\end{equation}
Compare \eqref{5} and \eqref{55}, we get the conclusion.
\end{proof}

\begin{rmk}
From the foregoing discussion, the method in this paper dealing with the optimal control problem of NSFDEs is consistent with the traditional one dealing with SDEs. However, when the state equation of the optimal problem behaves more generally than semi-martingale, the traditional one is no longer applicable.
\end{rmk}

\begin{rmk}
For more complex case, such as the general cost function and non-convex control set, the optimal control problem in this paper should be discusses further.
\end{rmk}

\bibliographystyle{siam}

\end{document}